\documentclass[12pt,a4paper]{amsart}

\RequirePackage[OT1]{fontenc}
\RequirePackage{amsthm,amsmath,amsfonts}
\RequirePackage[numbers]{natbib}
\RequirePackage[colorlinks,citecolor=blue,urlcolor=blue]{hyperref}
\newtheorem{thm}{Theorem}[section]

\newtheorem{lem}[thm]{Lemma}
\newtheorem{prop}[thm]{Proposition}
\theoremstyle{definition}

\newtheorem{defn}[thm]{Definition}
\theoremstyle{remark}

\numberwithin{equation}{section}

\usepackage{tikz}

\usepackage[lmargin=3cm,rmargin=3cm,bmargin=3cm,tmargin=3cm]{geometry}

\def\R{\mathbb R}
\def\j{\mathbf j}
\def\n{\mathbf n}

\title[Confidence nets based on finite reflection groups]{``Building'' exact confidence nets}

\author[Francis]{Andrew R. Francis}
\address{Centre for Research in Mathematics\\
Western Sydney University\\
Australia}
\email{a.francis@westernsydney.edu.au}
\thanks{ARF thanks the Australian Research Council for funding via FT100100898.}

\author[Stehlik]{Milan Stehlik}
\address{Institut f\"ur Angewandte Statistik\\
University of Linz\\
Austria, and\\
Departamento de Matem\'atica\\
Universidad T\'ecnica Federico Santa Mar\'{\i}a\\
Casilla 110-V\\
2360102 Valpara\'iso\\
Chile
}
\email{milan.stehlik@jku.at}

\author[Wynn]{Henry P. Wynn}
\address{Department of Statistics\\
London School of Economics\\
UK}
\email{h.wynn@lse.ac.uk}


\begin{document}
\maketitle

\begin{abstract}
Confidence nets, that is, collections of confidence intervals that fill out the parameter
space and whose exact parameter coverage can be computed, are familiar in nonparametric statistics.
Here, the distributional assumptions are based on invariance under the action of a finite reflection
group. Exact confidence nets are exhibited for a single parameter, based on the root system
of the group. The main result is a formula for the generating function of the
coverage interval probabilities. The proof makes use of the theory of ``buildings" and
the Chevalley factorization theorem for the length distribution on Cayley graphs of finite reflection groups.

\end{abstract}

\section{Introduction}\label{sec:intro}
It is well known, and usually attributed to Wilks~\cite{wilks1948order}, that the order statistics from a random sample
provide nonparametric confidence intervals for percentiles from a distribution: every interval formed by the
order statistics covers a given percentile with a computable probability. For the median the probabilities
are of binomial form. We shall refer to the situation in which the set of coverage intervals cover the real line
and the coverage probability of each interval is computable as a {\em confidence net}.

An interesting example  is given by Hartigan \cite{hartigan1969using,hartigan1975necessary} for the median, given an independent sample from a distribution symmetric about the median. There, the net is based on all sub-sample means: for a sample $\{y_i\mid i \in \n = \{1,\ldots, n\}\}$ and $S \subset \n$, a  subsample mean is $\frac{1}{|S|} \sum_{i \in S} y_i$, in which each of the $2^n$ intervals has coverage probability $\frac{1}{2^n}$. 
Hartigan's typical value theorem~\cite{hartigan1969using} is the basis for random subsampling, namely a resampling plan  to construct confidence intervals for the centre of a symmetric distribution on  a real line.
Atkins and Sherman \cite{atkins1992sets}  derived a group-theoretic condition on a set of subsamples of a random sample from a continuous random variable symmetric about zero to be sufficient to provide typical values for zero. With the current interest in very large data sets, subsampling from complex data can be viewed
as a natural solution to the computational issues. 
While many methods have been devised to provide unbiased and efficient estimation of average quantiles, to our knowledge no such method exploits invariance under the action of a finite reflection group.
Knowledge of invariance provides an omnibus method for constructing covering nets, which cannot be obtained by inverting  selected nonparametric multivariate rank tests (see~\cite{jurevckova2012nonparametric}).  Hartigan's work 
has also had impact in the theory of the bootstrap and resampling (see Efron~\cite{efron1982jackknife}, Efron and Tibshirani~\cite{efron1986bootstrap}).

Another example is the set of intervals formed by pairwise means, $\{\frac{y_i+y_j}{2}\}$, sometimes called Walsh averages. These are the basis for one version of the Hodge-Lehmann estimator for a mean \cite{hodges1963estimates}, which is the empirical median of all pairwise means (including the single observation). There are strong connections to the Wilcoxon signed rank (sum) test where the same generating function as derived for the group of type $B_n$ in this paper is used in the computation of critical values~\cite{mitic1996critical,van1999symbolic}. Indeed, the current paper could be represented as a group theoretic generalisation of the generating function approach of these papers, or, given the duality between testing and confidence intervals, as a way to invert certain permutation tests (see for example Trichler~\cite{tritchler1984inverting}).

We first give an account of a general construction of nonparametric confidence interval nets and then specialise to the case of finite reflection groups, showing the relation to the root systems of the groups.  Finite reflection groups have been classified completely up to isomorphism, and via this classification are also known as finite Coxeter groups, which also  have a purely algebraic definition based on their presentations.  We will not elaborate on this classification here but refer the reader to~\cite{Hum90} or~\cite{abramenko2008buildings}.  Subsections~\ref{sec:typeB} and~\ref{sec:typeD} cover in some detail the case of the finite Coxeter groups of type $B_n$  (the hyperoctahedral groups) and type $D_n$.  It turns out that in the case of type $B_n$ the interval boundaries are the pairwise means, mentioned above, together with the single observations. In the case of type $D_n$ they are the pairwise means, but excluding the single observations. The generating functions turn out to be familiar from the theory of partitions in number theory.

In Section~\ref{sec:gen.case} the main result of the paper is given, namely a  generating function for the interval probabilities for a general finite Coxeter group (with one exception). 
{Specifically, we show (Theorem~\ref{thm:main}) that the frequency distribution for the intervals of the confidence net based on a (almost any) finite irreducible Coxeter group is given by the generating function
$$G(q) =  \frac{\prod_{j=1}^m (1-q^{d_j})}{\prod_{i=1}^n(1-q^j)},$$
where $d_1, \ldots , d_m$ are the basic invariant degrees of  the group.
As an example, when the group is the Coxeter group of type $B_2$, the generating function is $G(q)=1 + 2q + 2q^2 +2q^3+q^4$.  This indicates that over the five intervals, the relative (coverage) probability of the parameter $\theta$ being in one of the middle three intervals is twice that of it being in one of the extremal intervals.  Details of this example and others appear in Sections~\ref{sec:cones} and~\ref{sec:gen.case}.
}

The proof {of this result} is given in Section~\ref{sec:proof} and relies on showing that the probabilities are derived from the Coxeter length function for the quotient of the Coxeter group by the symmetric (permutation) group (the finite Coxeter group of type $A_n$). To translate the geometry of the confidence net {into group theory} requires the theory of buildings (given in Section~\ref{sec:chambers}) and specifically the mapping
of intervals into ``chambers'' and the full collection of intervals, nets, into  ``galleries'' formed by chambers.  Because of the strong {links with} group theory, we also put this paper forward as a contribution to the rapidly developing area of ``algebraic statistics", in which there has been renewed interest in permutation tests;  see for example Morton et al~\cite{morton2009convex}.

There is a long tradition of the study of ``statistics'' (also called indices) such as the length function, on groups. For example, Reiner~\cite{reiner1993signed} studied the extension of such statistics from the symmetric groups to type $B_n$. Adin and Roichman~\cite{adin2001flag} defined a new index called the flag major index whose length was equidistributed  in the type $B$ case. They used this to study group actions on polynomial rings. Geometric distance problems in genomic rearrangements can be reduced to Coxeter length problems~\cite{francis2013algebraic,egri2014group}. In statistics, Diaconis~\cite[Chapter 4C]{diaconis1988group}, makes the connection between length distributions and non-parametric tests.

Our general formula (Theorem~\ref{thm:main}) agrees with that for $B_n$ and $D_n$ already derived in Sections~\ref{sec:typeB} and~\ref{sec:typeD}, using a counting argument on the raw inequalities describing the cones of the groups. The generating functions for the exceptional groups $E_6$, $E_7$, $E_8$ and for the groups of type $A_n$ are given as examples after the main proof. The net in the case of $E_8$ has a remarkable $93$ cells. 
While these {$E_n$} cases can only be used for sample sizes $6,7$ and $8$ respectively, they are nonetheless of independent interest. 
The paper concludes with short sections on an example not in the group class, and some simple asymptotics.

\section{Confidence nets}\label{sec:confnets}
Let $Y$ be a random $n$-vector with probability density function $f(y, \theta)$, where $\theta$ is an unknown $k$-dimensional parameter. 
{For most of this paper we will study the case $k=1$, but begin in this Section with the general set-up.}
We assume that $Y$ can be transformed by a measurable transformation $T(y,\theta)$, typically $\theta$-dependent, to a random variable $Z$:
$$Z = T(Y, \theta),$$
which is  also $n$-dimensional and has a distribution  some of whose properties are known, independently of $\theta$.

Assume there exists a finite collection of sets $\{C_i,\;i=1, \ldots, m\}$, such that
\begin{enumerate}
\item $\bigcup C_i =  \R^n$,
\item The measure with respect to $Z$ of any intersection $C_i \cap C_j,\;  i \neq j$, is zero,
\item $\mbox{prob} \{Z \in C_i\} = \alpha_i,\; i=1, \ldots, m.$
\item The $\alpha_i$ are positive, do not depend on $\theta$, and $\sum_{i=1}^m \alpha_i =1$.
\end{enumerate}
Define, for fixed $y$
$$S_i(y) = \{ \theta : Z \in C_i\}.$$
Thus, $S_i(y)$  is the inverse of the function $T(y,\theta)$ for  fixed $y$ and
$$\mbox{prob} \left\{S_i(y) \ni \theta\right\} = \alpha_i.$$
We should note that typically $k$ (the dimension of $\theta$) is very much smaller than $n$.

A confidence net is based on the following coarsening in the description of the coverage sets $S_i$ using geometric considerations.  Suppose that  there are $N$ random sets $\{U_j(Y),\, j = 0, \ldots, N-1\}$ in $\theta$-space whose intersections cover $\theta$ with zero probability, $\bigcup_{j=0}^{N-1} U_j =  \R^k$, and such that for any $i = 1, \ldots, m$, there is a mapping $j = u(i)$ such that
$$S_i(y) = U_j(y),$$
and moreover that every $U_j$ can be obtained in this way. The mapping $u(\cdot)$ is typically a many-to-one mapping, and given any
$j$ we can define {the} inverse:
$$u^{-1}(j) = \{i: j = u(i)\}.$$
This implies that the $U_j(Y)$ are themselves (random) coverage sets with coverage probabilities
\begin{align*}
p_j & = \mbox{prob}  \{U_j  (Y) \ni \theta\} \\
    & = \sum_{i \in u^{-1}(j)} \alpha_i
\end{align*}
for $j= 0, \ldots, N-1$. Note also that since $\sum_{j=0}^{N-1} \alpha_i = 1$, we have $\sum_{i=0}^{N-1} p_i =1$. We refer to the
set $\{U_j(Y)\}$ as an {\em exact confidence net}.  We summarise this in the following definition.

\begin{defn} For a parametric statistical model with random variable $Y$ (possibly multivariate), an \emph{(exact) confidence net} is a collection of data dependent sets $U_j(Y), j=0, \ldots, N-1$ whose union is the whole parameter space, and such that the probability
that $U_j(Y)$ covers the parameter $\theta$ is a known quantity $\alpha_j$, $j=0, \ldots, N-1$ and such that any intersection of the $U_j$ covers $\theta$ with probability zero.
\end{defn}


Here we take a classical statistical approach to coverage nets. Thus, the notion is that the user {\em declares} the sets $U_j$. 
{Theories} of inference based on collections of coverage sets can be thought of as part of a well developed theory of belief {functions} based on upper and lower probabilities and the theory of random sets based on Choquet capacities (see~\cite{dempster1967upper,wasserman1990bayes}). In terms of the former, a coverage net is essentially a theory of random sets in which the upper and lower probabilities coincide, and in which the Choquet capacity functional is additive over the $\sigma$-algebra of unions of sets. That is to say, for the sets $\{U_j\}$ we have for $i \neq j$,
$$\mbox{prob} \{ U_i \cup U_j \ni \theta \} = \mbox{prob} \{ U_i \ni \theta \} +  \mbox{prob} \{ U_j \ni \theta \},$$
and so on.

\section{Reflection groups and cones}\label{sec:cones}

Let $Y$ be an $n$-dimensional random vector and $\theta$ be a univariate parameter ($k=1$). Define
$$
Z(Y,\theta)=(Y_1-\theta, Y_2 - \theta, \ldots, Y_n-\theta)^T.
$$
Let $G$ be a finite reflection group acting on $\R^n$ and let $\{C_i, \; i=1, \ldots, m\}$, where $m=|G|$, be the collection
of cones in $\R^n$ that are the transformations under $G$ of the fundamental cone $C_1$. Our key condition, corresponding to condition (3) in Section~\ref{sec:confnets}, is that every such cone has the same probability content with respect to the distribution of $Z$:
$$
\mbox{prob} \left\{ Z \in C_i\right\}  = \frac{1}{m}, \; \; i= 1, \ldots, |G|.
$$

Each statement $\{y \in C_i\}$  yields a statement $\theta \in U_j$. To find $\{U_j\}$ we need to provide the mapping $i \mapsto j=u(i)$. The $U_j$ are intervals and it is enough to give their endpoints. Then, for each $j$ we can define the count
$N_j = | u^{-1}(j)|$, namely the number of cones giving $U_j(y)$. Then under our assumptions
$$p_j = \mbox{prob}\left\{U_j \ni \theta\right\} = \frac{N_j}{m}, \quad j=0, \ldots, N-1.$$

Fortunately, although the orders of the groups can be very large, the geometry of finite reflection groups can be understood in terms of their root systems, and
the number of roots is orders of magnitude smaller.  In all that follows the counts $N_j$ will have a factor of $n!$, the order of the symmetric group, and it is somewhat neater, therefore, to work with $n_j = \frac{N_j}{n!}$. The main objective of this paper is to find, for different groups,
the generating function for the $\{n_j\}$:
$$G(q) = \sum_{j=0}^{N-1} n_j q^j.$$

Every finite reflection group is defined by its {\em roots}. These are vectors $\{a_j\}$ that define the perpendiculars to the defining hyperplanes
 $$H_j = \{x: a_j^T x =0 \}$$
 forming the walls of the cones $C_i$. Roots are identified with half-spaces and therefore come in pairs: $\pm a_j$, which are important in the classification of these groups (for more details on root systems and the classification of finite reflection groups, see for instance~\cite{Hum90}).

Before we proceed to a general approach, in the next two subsections  we will use an elementary discussion of inequalities and a counting argument to derive the generating functions $G(q)$, in two cases.

\subsection{The hyperoctahedral groups: type $ B_n$}\label{sec:typeB}

The group of type  $B_n$ (we will refer to the group simply as ``$B_n$") operates on points $z \in \R^n$ by permutation and sign change of the coordinates. It has order $2^nn!$. Its fundamental cone
$C_1$ is (by convention) given by
$$
z_1 \geq z_2 \geq \cdots  \geq z_n \geq 0,
$$
and it  has fundamental roots given by each of the inequalities above. In standard notation the roots are
$$
\{e_1-e_2, e_2-e_3, \ldots, e_{n-1} - e_n, e_n\},
$$
where the $e_i$ are unit vectors. The fundamental roots are thus
$$
(1,-1,0, \ldots, 0)^T, (0,1,-1,0,\ldots,0)^T, \ldots, (0, \ldots, 0, 1)^T.
$$
It is important to repeat that all other roots come from transformation of these roots under the group.
Each cone $C_i$ is obtained by transformation of $C_1$ under a suitable group element. We can describe
the cones compactly by inequalities:
$$
\pm z_{\pi(1)} \geq \pm z_{\pi(2)} \geq \cdots \pm z_{\pi(n)} \geq 0, \label{inequals}
$$
where $\pi=(\pi(1),\dots,\pi(n))$ ranges over all $n!$ permutations of $\{1, \ldots, n\}$.

Substituting $z_i = y_i - \theta$ for $i=1, \ldots, n$,  we have
$$
\pm (y_{\pi(1)} - \theta) \geq \pm (y_{\pi(2)} - \theta) \geq \cdots \geq \pm (y_{\pi(n)} - \theta) \geq 0. \label{inequals}
$$
Thus, every cone $C_i$ is defined by a set of inequalities for $\theta$, each of which
yields a $U_j$ interval for $\theta$.

Consider the case $B_3$. First fix the order $z_1, z_2,z_3 $. 
The set of inequalities which describes the cones for this order is
\begin{align*}
 +(y_1-\theta) \geq +(y_2-\theta) \geq +(y_3-\theta) & \geq 0 \\
 +(y_1-\theta) \geq +(y_2-\theta) \geq -(y_3-\theta) & \geq 0 \\
 +(y_1-\theta) \geq -(y_2-\theta) \geq +(y_3-\theta) & \geq 0 \\
 -(y_1-\theta) \geq +(y_2-\theta) \geq +(y_3-\theta) & \geq 0 \\
 +(y_1-\theta) \geq -(y_2-\theta) \geq -(y_3-\theta) & \geq 0 \\
 -(y_1-\theta) \geq +(y_2-\theta) \geq -(y_3-\theta) & \geq 0 \\
 -(y_1-\theta) \geq -(y_2-\theta) \geq {+(y_3-\theta)} & \geq 0 \\
 -(y_1-\theta) \geq -(y_2-\theta) \geq -(y_3-\theta) & \geq 0. \\
\end{align*}
We are interested in the index of the interval {that covers} $\theta$, that is, its \emph{position} among the subset means, {because this determines the interval $U_j$ for $\theta$}.  There are $\binom{3}{ 2} + 3=6$ subset means, namely the pairwise means $\frac{y_i+y_j}{2}$ ($i\neq j$) and the individual $y_i$'s.  For instance, the first row of inequalities above yields $\theta\leq y_1,y_2,y_3$, and consequently $\theta$ is also less than each of the pairwise means, placing it in the zero-th position in interval $U_0$.  A less trivial example is the fourth row of inequalities, whose manipulations  yield 
\[ y_1,\frac{y_1+y_2}{2},\frac{y_1+y_3}{2}\le \theta \le y_2,y_3,\frac{y_2+y_3}{2},
\]
{so that $\theta$ is covered by $U_3$.  Note that other inequalities, such as $y_3\le y_2$, follow from the same row of the list of inequalities above, but these do not affect the coverage interval for $\theta$.}
The fifth row of inequalities also places $\theta$ in third position.  Repeating this for each row, we obtain  $\theta$, respectively, in intervals: $U_0,U_1,U_2, U_3,U_3,U_4,U_5,U_6$ (noting the double representation of the middle interval). We obtain the same distribution
for all permutations $\pi$. 
This gives the $n_i$ count
as $1,1,1,2,1,1,1,$ with generating function
$$G(q) = 1+q+q^2+2q^3+q^4+q^5+q^6.$$

For $B_4$ a similar calculation yields $4! 2^4 = 384$ sets of inequalities, in blocks of $2^4 = 16$, one block of inequalities for each permutation, as for $B_3$. There are 11 intervals formed by the 10 values $\{ y_i, \frac{y_i + y_j}{2}, i,j=1,\ldots, 4,\; i \neq j\}$. The $n_i$ count is $1,1,1,2,2,2,2,2,1,1,1$ {(summing to 16)} with generating function
$$G(q) = 1+ q + q^2 + 2q^3 +2q^4+2q^5+2q^6 + 2q^7 + q^8 + q^9 + q^{10}.$$

From these examples we see how to evaluate the vector $(n_0,\dots,n_{N-1})$ for $B_n$. {For} any line of inequalities (cone)
the $k$-th interval  covers $\theta$ if and only if there are exactly $k$ elements from the set of possible boundaries $\{y_i, \frac{y_i+y_j}{2},\; i < j\}$ less than or equal to $\theta$. Thus, for our initial permutation and a particular cone:
$$
\begin{array}{rcl}
n_k & = & \left| \{i: y_i \leq \theta \}\right|\ +\ |\{ (i,j): i \leq j; \; \frac{y_i+y_j}{2} \leq \theta\}|.
\end{array}
$$
Now consider which sign combinations {on the $(y_i-\theta)$} lead to a contribution to $n_k$. The possibilities
are:
\begin{enumerate}
\item  $y_i \leq \theta$: a single $-$ at position $i$.
\item $\frac{y_i+y_j}{2} \leq \theta$: a pair $ -, + $ in positions $i< j$, respectively.
\item $\frac{y_i+y_j}{2} \leq \theta$: a pair $ -, - $ in positions $i< j$, respectively.
\end{enumerate}
Define indicator functions which capture the sign combination:  $x_i=1,0$ for $-,+$ in the $i$-th position respectively. We can then
set up a counting function to capture $n_i$:
\begin{align*}
\psi(x_1,\ldots, x_n) & = \sum_{i=1}^n x_i + \sum_{i,j, i < j}^n x_i(1-x_j) + \sum_{i<j}^n x_i x_j \\
		& = \sum_{i=1}^n i x_i.
\end{align*}

Now, considering the $x_i$ as independent Bernoulli
random variables, each  $\{i x_i\}$ is independent with support $\{0, i\}$ and probability generating function
$\frac{1}{2}(1+q^i)$, by convolution, and multiplying by $2^n$, we see that
$$G_n(q) = (1+q)(1+q^2) \cdots (1+q^n),$$
which we confirm in the cases $n=3,n=4$, above.

\subsection{The groups of type $D_n$}\label{sec:typeD}
The group of type $D_n$ (which we again will call simply the group $D_n$) is the group of permutations with an even number of sign changes, and has order $2^{n-1} n! $.
It has fundamental cone
$$z_1 \geq z_2 \geq \cdots  \geq z_n, \; z_{n-1}+z_n \geq 0.$$
The roots are $$\{e_1-e_2, e_2-e_3, \ldots, e_{n-1} - e_n, e_{n-1}+e_{n}\},$$
giving
$$(1,-1,0,\ldots,0)^T, (0,1,-1,0,\ldots,0)^T, \ldots, (0, \ldots, 0, 1,-1)^T, (0,\ldots, 0, 1,1)^T.$$
In the case $n=4$ the inequalities are
$$z_1\ge z_2\ge z_3\ge z_4; \; z_3\ge -z_4,$$
giving
$$y_1-\theta \ge y_2-\theta \ge y_3-\theta\ge y_4-\theta;\;  y_3-\theta\ge -(y_4-\theta).$$
 Now $D_4$ allows permutations of the coordinates as well as {\em even} numbers of sign changes.  Therefore the possible signs are as follows (the vertical and horizontal lines will be explained shortly):
\begin{center}
$$
\begin{array}{ccccc}
+&+&+& \vline &+\\
-&-&+&\vline& +\\
-&+&-&\vline&+\\
+&-&-&\vline&+\\
\hline
-&+&+&\vline&-\\
+&-&+&\vline&-\\
+&+&-&\vline&-\\
-&-&-&\vline&-
\end{array}
$$
\end{center}
The second line, for example,  gives
$$-(y_1-\theta)\ge -(y_2-\theta)\ge y_3-\theta\ge y_4-\theta;\quad y_3-\theta\ge -(y_4-\theta),$$
from which we deduce
\begin{eqnarray*}
\theta & \geq &\frac{y_1+y_3}{2},\frac{y_1+y_4}{2},\frac{y_2+y_3}{2},\frac{y_2+y_4}{2},\frac{y_1+y_2}{2}; \\
\theta & \leq & \frac{y_3+y_4}{2}.
\end{eqnarray*}
The inequality $\theta \geq \frac{y_1+y_2}{2}$ is found by first noting that $\theta \geq y_2 \geq y_1$.
It is tempting to include the singletons $y_i$ in the set of boundary points, but not all $y_j$ can be determined in this way
which means that intervals using the $y_i$ are not fully computable.

Following the last remark, we determine coverage of $\theta $ given by all pair means $\frac{y_i+y_j}{2},\; i<j$.

We shall need to account for the following possibilities using slightly more complicated rules than for $B_n$:
\begin{enumerate}
\item  $y_i \leq \theta$: a single $-$ at position $i$, for $i=1,\;\dots,n-1$. This is used to help place the pair-means.
\item  $\frac{y_i+y_n}{2} \leq \theta$: $-$ in position $i$ and  $-$ in position $n$ for $i=1,\ldots, n-1$,
\item  $\frac{y_i+y_j}{2} \leq \theta$ for $1\leq i < j \leq n$: a pair $ -, +$ in positions $i,j$ respectively,
\item  $\frac{y_i+y_j}{2} \leq \theta$ for $1\leq i < j \leq n-1$: a pair $ -, -$ in positions $i,j$ respectively. This follows
by noting that that $y_i, y_j \leq \theta$, as in rule (1), above.
\end{enumerate}
We can split rule (3), above, into two cases: when the $+$ is in position $1,\ldots, n-1$; and when the $+$ is in position $n$ (hence the vertical line in the preceding figure). The latter can be combined with rule 2 to give rule 1. Again we take the indicator with $x_i=1$ for $-$ and $x_i=0$ for $+$ and our counter is
\begin{align*}
\psi(x) & =  \sum_{i=1}^{n-1}x_i + \sum_{i<j}^{n-1} x_i(1-x_j) + \sum_{i<j}^{n-1}x_i x_j\\
  & =   \sum_{i=1}^{n-1}i x^i.
\end{align*}
Thus, the generating function for the $n_i$, again using a convolution argument, is
$$G_{n-1}(q) = (1+q^2)(1+q^3)\cdots (1+q^{n-1}).$$

In the following section we re-derive these generating functions in a general framework, using the theory of buildings.

\section{The main result}\label{sec:gen.case}
As mentioned in the introduction we will use some theory developed around the concept of indices (sometimes call ``statistics") attached to
an element $g$ of a group $G$. MacMahon~\cite{macmahon1915combinatorial} discussed, for the symmetric group, {\em descent, excedance, length} and the {\em major index}. Authors are often interested in the frequency of the distinct values of an index as $g$ ranges over the whole group, and there are strong combinatorial results, going back to MacMahon, showing that one index has the same distribution as another, even though the actual indices (as mappings) are different. This work is relevant for us because (i)  we have a special index which is the value $j$ of our interval  $U_j$ of the net construction; (ii) generating functions play an important role; and (iii) the study of such indices is being extended to finite reflection groups such as $B_n$ and $D_n$.

A starting point for the construction of these indices is the Cayley graph of a group. If  $S$ is the set of generators of our group,
then the Cayley graph is a graph $(E,V)_G$ where each vertex $v_g\in V$ is labelled by a group element $g \in G$ and each edge $e_{g, h}$ by a single right multiplication by a generator $s\in S$ : $h = g s$; only generators may be used. The length, $l(g)$ of a group element $g \in G$ is the length of the minimal path on the graph from the identity $e$ to $g$, when each edge counts unity:
$$ l(g) = \min\{k \geq 0 : g = s_{i_1} s_{i_2}  \cdots s_{i_k},\text{ for } s_{i_1}\in S\}.$$

The Cayley graph for the group $B_2$ has two generators which we may take (on the $(z_1,z_2)$ plane) as  (i) $s_1$ the reflection in the line $z_1=z_2$ and (ii) $s_2$ the reflection in the $z_1=0$ axis. The Cayley graph and corresponding lengths are given in Figure~\ref{fig:cayley.B2}.

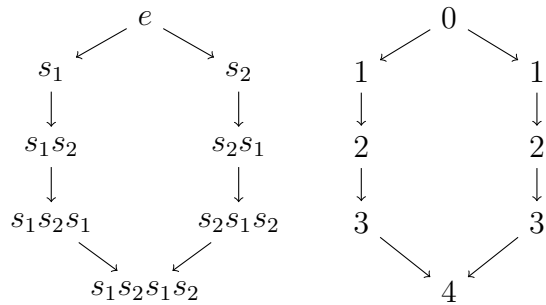
\begin{figure}[ht]
\begin{tikzpicture}
\draw (0,0) node (e) {$e$};
\draw node[below left of=e,left=2mm] (s1) {$s_1$}; %
\draw node[below right of=e,right=2mm] (s2) {$s_2$}; %
\draw node[below of=s1] (s12) {$s_1s_2$};
\draw node[below of=s2] (s21) {$s_2s_1$};
\draw node[below of=s12] (s121) {$s_1s_2s_1$};
\draw node[below of=s21] (s212) {$s_2s_1s_2$};
\draw node[below of=e,below=23.5mm] (s1212) {$s_1s_2s_1s_2$};
\draw[->](e)--(s1);
\draw[->](e)--(s2);
\draw[->](s1)--(s12);
\draw[->](s2)--(s21);
\draw[->](s12)--(s121);
\draw[->](s21)--(s212);
\draw[->](s121)--(s1212);
\draw[->](s212)--(s1212);

\draw (4,0) node (0) {$0$};
\draw node[below left of=0,left=2mm] (1) {$1$};
\draw node[below right of=0,right=2mm] (1') {$1$};
\draw node[below of=1] (2) {$2$};
\draw node[below of=1'] (2') {$2$};
\draw node[below of=2] (3) {$3$};
\draw node[below of=2'] (3') {$3$};

\draw node[below of=0,below=23.5mm] (4) {$4$};
\draw[->](0)--(1);
\draw[->](0)--(1');
\draw[->](1)--(2);
\draw[->](1')--(2');
\draw[->](2)--(3);
\draw[->](2')--(3');
\draw[->](3)--(4);
\draw[->](3')--(4);
\end{tikzpicture}
\caption{The Cayley graph for $B_2$ with its elements' lengths. }\label{fig:cayley.B2}
\end{figure}

The length frequency distribution is $\{f_0, f_1,\dots, f_m\}$ where $f_j = \#\{g : l(g) = j,\, g \in G\}$ and $m$ is the diameter of the group. The generating function for the length frequencies is
$$G(q) = \sum_{j=0}^m f_jq^j.$$
We can compute $G(q)$ using the Chevalley factorization theorem (see for instance~\cite[Section 3.15]{Hum90}):
\begin{thm}\label{thm:ourgroup}
Let $W$ be an irreducible Coxeter group. Then the length generating function is
$$G_W(q) = \prod_{j=1}^m \frac{1-q^{d_j}}{1-q},$$
where $d_1, \ldots , d_m$ are the basic invariant degrees of  the group.
\end{thm}

Table~\ref{tab:degrees}, taken from \cite[Section~3.7]{Hum90}, lists the degrees for the crystallographic Coxeter groups.  The polynomial $G_W(q)$ is known as the \emph{Poincar\'e polynomial} of $W$.

\begin{table}
$$
\begin{array}{l|l}
  \mbox{Type}      & d_1, d_2, \ldots \\ \hline
   A_n      & 2,3, \ldots, n+1 \\
   B_n, C_n & 2,4, \ldots , 2n-2, 2n \\
   D_n      & 2,4,6,\dots, 2n-2,n \\
   E_6      & 2,5,6,8,9,12 \\
   E_7      & 2,6,8,10,12,14,18 \\
   E_8      & 2,8,12,14,18,20,24,30 \\
   F_4      & 2,6,8,12 \\
   G_2      & 2,6
\end{array}
$$
\caption{Degrees for the crystallographic Coxeter groups}\label{tab:degrees}
\end{table}

With our running $B_2$ example,
\begin{align*}
G_W(q)  & = \frac{(1-q^2)(1-q^4)}{(1-q)^2}
      \ = 1 + 2q + 2q^2 +2q^3+q^4,
\end{align*}
giving the frequencies $(1,2,2,2,1)$, as expected from the graph in Figure~\ref{fig:cayley.B2}.

The length distribution for the symmetric group is
$$G_{S_n}(q)= \prod_{i=1}^n \frac{1-q^i}{1-q}.$$
In the main theorem, which follows, the formula is obtained by dividing the generating
function for the length distribution of our group given in Theorem~\ref{thm:ourgroup}, by that for the symmetric group.

\begin{thm}\label{thm:main}
The generating function for the frequency distribution for the intervals of the confidence net based on a finite irreducible Coxeter group $G$ of any type except $F_4$ is given by
$$G(q) =  \frac{\prod_{j=1}^m (1-q^{d_j})}{\prod_{i=1}^n(1-q^j)},$$
where $d_1, \ldots , d_m$ are the basic invariant degrees of  the group.
\end{thm}

The exclusion of $F_4$ in the theorem statement is necessary because the result depends on the symmetric group being a maximal parabolic subgroup of $G$, which holds in all cases except $F_4$ (see for instance~\cite[Appendix A]{geck2000characters}).  We leave the calculation of the generating function for $F_4$ as an exercise along the lines of the examples in Sections~\ref{sec:typeB} and~\ref{sec:typeD}.

Before we prove this theorem (in Section~\ref{sec:proof}), we demonstrate with two examples that its results agree with those calculated in Sections~\ref{sec:typeB} and~\ref{sec:typeD}.

For $B_n$ the formula in Theorem~\ref{thm:main} gives
\begin{eqnarray*}
\frac{G_W(q)}{G_n(q)} \ = \  \frac{\prod_{i=1}^n (1-q^{2i})}{\prod_{i=1}^n(1-q^j)} 
                      \ = \  \prod_{i=1}^n (1+q^{_j}),
\end{eqnarray*}
as expected. Note that we have two ways of counting the number of intervals: the number of live roots, following Lemma \ref{lem:cones.ray}, and the degree of $G(s)$:
$$n + {n \choose 2} = \sum_{j=1}^n j.$$

For $D_n$ the formula is
\begin{eqnarray*}
\frac{G_W(q)}{G_n(q)} \ = \  \frac{\prod_{j=1}^{n-1}(1-q^{2j})(1-q^n)}{\prod_{i=1}^n (1-q^j)} 
                      \ = \ \prod_{j=1}^{n-1} (1+q^{_j}),
\end{eqnarray*}
again, as expected.

Before we proceed to the proof of Theorem~\ref{thm:main}, we need to introduce some more of the tools of the theory of buildings.

\section{Rays, chambers and Cayley graphs}\label{sec:chambers}

Returning to our construction from Section~\ref{sec:cones}, in vector notation we have
\begin{equation}\label{ray}
Z(y,\theta) = y - \theta {\bf j},
\end{equation}
where ${\bf j} = (1,1, \ldots, 1)^T$. For fixed $y$ we shall refer to the one dimensional affine subspace defined by (\ref{ray}), as $\theta$ varies, as the {\em ray} from $y$, denoted $E_y$:
$$E_y=\{z: z = y - \theta {\bf j},\; \theta\in\R\}.$$   
We have:
\begin{lem}\label{lem:cones.ray}
Let $\{C_i\}$ be the collection of all cones generated in the standard way by a finite reflection group, and let $y$ be a non-zero vector.
\begin{enumerate}
\item If $y$ is in general position (not lying in any defining hyperplane $H_i$) then the ray $E_y$ intersects the faces of a fixed number $N$ of the cones $C_i$ at values $\theta_1(y) < \theta_2(y) < \ldots < \theta_{N-1}(y)$.
\item Any $\theta_j(y)$ is given by
$$\theta_j(y) = \frac{a^Ty}{a^Ta},$$
for some positive root $a$ which is not orthogonal to ${\bf j}$.
\item $N-1$ is the number of roots not orthogonal to $\bf j$.
\end{enumerate}
\end{lem}
\begin{proof}
By elementary geometry, when $y$ is in general position the ray $E_y$ intersects every defining hyperplane $H_i$ exactly once except when  {\bf j} lies in an $H_i$, in which case it does not intersect that $H_i$. Because for any hyperplane $H_i$, its root $a_i$, by definition defines  the orthogonal subspace to $H_i$, the latter condition is equivalent to being orthogonal to $a_i$. Each cone has two intersection points except for the end cones when the intersection are at $\theta_1(y)$ and $\theta_{N-1}(y)$.  Part (2) follows since the intersection points satisfy: $a_j^T(y- \theta {\bf j}) = 0$.
\end{proof}
We refer to hyperplanes $H_i$ as being {\em live} if their roots are not orthogonal to ${\bf j}$ .
The intervals we require are
\begin{align*}
U_0 &=(-\infty, \theta_1(y)],\quad U_1 =[\theta_1(y), \theta_2(y)],\quad \ldots, \\
U_{N-2} & =[\theta_{N-2}(y), \theta_{N-1}(y)],\quad  U_{N-1}   =[\theta_{N-1}(y), \infty).
\end{align*}

To prove Theorem~\ref{thm:main}, we need to introduce some of the group-theoretic geometry behind it.
An excellent reference for further reading on this topic is~\cite[Chapter 1]{abramenko2008buildings}.

The {\em chamber graph} of a finite reflection group has cones (called chambers in this context) as vertices, with two cones having an edge if they share a common face. A path in the chamber graph is called a {\em gallery}: imagine a walk through chambers
with doors in the common wall (facet). With each edge given length unity, distance between chambers is defined (as for a Cayley graph) by the shortest distance between the chambers, and we call the corresponding gallery {\em minimal}.  A gallery is minimal if it does not cross any wall more than once (\cite{abramenko2008buildings}  Proposition 1.56). Since a straight line in general position (in an obvious sense) cannot cut any wall of a chamber more than once, it defines a minimal gallery. For both the Cayley graph and the chamber graph, $C_e$ is the cone corresponding to the identity element $e$ of the group, and we call this the fundamental cone.

Following Lemma~\ref{lem:cones.ray}, the ray $E_y$ defines a gallery that we denote $G_y$.  This gallery starts in the identity chamber $C_e$ and has length $N$ (there are $N$ chambers along it).  The index $j$ of a chamber $C_g$ yielding the interval
$$U_j = [\theta_j(y), \theta_{j+1}(y)]$$
is the distance in the gallery $G_y$ from $C_e$ to $C_g$.
Consequently, the number of cones $|u^{-1}(j)|
$ that map into a given index $j$ is the number of group elements of distance $j$ from $C_e$ along the gallery $G_y$.

Now consider reflections in the walls of a chamber. Suppose this chamber is a translation by $w$ of the fundamental chamber $C_e$, so that its faces are translations of the fundamental hyperplanes that are the faces of $C_e$.  If $H_s$ is a face of $C_e$ (for a generator $s$ of $G$), then it is translated by $w$ to $wH_s$, and reflection in this hyperplane corresponds to action by the reflection $wsw^{-1}$ (in general this is not a fundamental reflection).  Thus, reflection in the face $wH_s$ of $wC_e$ gives the chamber given by the left multiplication of $w$ by $wsw^{-1}$, namely $wsw^{-1}w=ws$, or $wsC_e$.   In other words we move from the chamber $wC_e$ to the chamber $wsC_e$.  Thus, movement along a gallery corresponds to \emph{right} multiplication by a generator.

As an aside, it is worth noting that the movement along the gallery by right multiplication provides a correspondence between the chamber graph and the Cayley graph, in which the movement along edges is given by left multiplication $w\to sw$.  The chamber graph, however, is the natural place for our results because it has a very direct link with the geometry.

\section{Proof of Theorem~\ref{thm:main}}\label{sec:proof}

Theorem~\ref{thm:main} states, in effect, that the distribution of distances of group elements along galleries defined by the rays $E_y$ is the same as the distribution of lengths of minimal coset representatives when the quotient of $G$ is taken by the symmetric group $S_n$.  This is because the numerator of this generating function is the Poincar\'e polynomial of the group, and the denominator is that of the symmetric group (see~\cite[Section 1.11]{Hum90} for more details).  For background reading on the theory of reflection groups and buildings there are many good sources, but we recommend in particular Abramenko and Brown~\cite{abramenko2008buildings}, Humphreys~\cite{Hum90}, and Kane~\cite{kane2001reflection}.

To prove Theorem~\ref{thm:main}, it suffices to show that the set of group elements along the galleries defined by the rays $E_y$ is precisely the set of minimal length coset representatives of $S_n$ in $G$, for $G$ of the types given in the theorem.  We prove this in Proposition~\ref{prop:lemmas}, below, but first a short lemma.

\begin{lem}\label{lem:roots.Sn.orth.j}
The roots from $S_n$ are all orthogonal to $\mathbf j=(1,1,\dots,1)^T$.
\end{lem}

\begin{proof}
Action by any element of $S_n$ fixes $j$; that is, reflection in any root from $S_n$ fixes $\mathbf j$, which means that the root must be orthogonal to $\mathbf j$.
\end{proof}

There are some well-known facts about Coxeter groups that we refer to in what follows, gathered in the  Lemma below:

\begin{lem}\label{lem:coxeter.facts}
Let $W$ be a finite Coxeter group, and $W'$ a parabolic subgroup of $W$.
\begin{enumerate}
	\item The minimal length elements of the cosets of $W'$ in $W$ are unique.\label{fact:min.elts.unique}
	\item The minimal length elements of the cosets of $W'$ in $W$ add in length when multiplied by any element of $W'$. \label{fact:min.coset.elts.add}
	\item If $W'$ is a parabolic subgroup of $W$, then the length distribution of distinguished coset representatives of $W'$ in $W$ is given by $G_W(t)/G_{W'}(t)$, where $G_W(t)$ and $G_{W'}(t)$ are the respective Poincar\'e polynomials. \label{fact:poincare.poly.quotient}
	\item The length of the longest word $w_0$ in $S_n$ is the number of positive roots in the root system of $S_n$.  \label{fact:longest.word.positive.roots}
\end{enumerate}
\end{lem}

\begin{proof}
These {statements} are all given in various texts, but in particular all are in~\cite{Hum90}: for \eqref{fact:min.elts.unique} and~\eqref{fact:min.coset.elts.add} see~\cite[Section~1.10]{Hum90}; for~\eqref{fact:poincare.poly.quotient} see~\cite[Section 1.11]{Hum90}; and for~\eqref{fact:longest.word.positive.roots} see~\cite[Section 1.8]{Hum90}.
\end{proof}

\begin{prop}\label{prop:lemmas}\label{lem:gal.coset.prec}\label{lem:cosets.meet.galleries}\label{lem:last.chamber}\label{lem:gallery.min.coset.reps}
For chambers defined by the action of a finite Coxeter group on $\R^n$, let the gallery $G_y$ be the series of adjacent chambers beginning with $C_e$, defined by the ray $E_y$ where $y$ is some point in $C_e$.
\begin{enumerate}\renewcommand{\theenumi}{\roman{enumi}}
	\item The group elements labelling chambers in $G_y$ are all minimal length $S_n$-coset representatives.  \label{prop:part2}
	\item Every $S_n$-coset has its minimal length element appearing on a gallery $G_y$ for some $y$. \label{prop:part1}
\end{enumerate}
\end{prop}

\begin{proof}
First note that $S_n$ is a parabolic subgroup of every finite Coxeter group $G$ except $F_4$, and we consider $G$ acting on $\R^n$~\cite[Appendix A]{geck2000characters}.
Explicitly: $S_n$ is a parabolic subgroup of the groups of types $A_n$, $B_n$ and $D_n$ that we consider acting on $\R^n$, and of the groups of types $E_n$ for $n=6,7,8$ (acting on $\R^n$); $S_3$ and $S_4$ are parabolic subgroups of the groups $H_3$ and $H_4$ acting on $\R^3$ and $\R^4$ respectively; and $S_2$ is (rather trivially) a parabolic subgroup of each of the dihedral groups $I_2(m)$ acting on $\R^2$.

Part~\ref{prop:part2}.

We begin by showing that the element corresponding to the \emph{last} chamber in the gallery is in the same coset as the longest word in the group $G$.
The length of the longest word $w_0$ in $S_n$ is the number of positive roots in the root system of $S_n$ (Lemma~\ref{lem:coxeter.facts}\eqref{fact:longest.word.positive.roots}), which is the number of roots orthogonal to $\j$ (Lemma~\ref{lem:roots.Sn.orth.j}).  Let $\mathcal C$ be the $S_n$-coset containing $w_0$.  Because minimal coset elements add in length with any element of $S_n$ (Lemma~\ref{lem:coxeter.facts}\eqref{fact:min.coset.elts.add}), the minimal coset representative in $\mathcal C$ must have length the number of positive roots \emph{not} orthogonal to $\j$.

An element of $\mathcal C$ appears in some gallery $G_y$, for some $y$, by Lemma~\ref{lem:cosets.meet.galleries}.  The number of chambers in each gallery is the number of roots not orthogonal to $\j$.  Since crossing each hyperplane from one chamber to the next along the gallery adds at most 1 in length, the longest element on any gallery is at most length the number of roots not orthogonal to $\j$.  Therefore the longest word of the group must be in a coset whose minimal length element is precisely the number of roots not orthogonal to $\j$.  This can only be the last chamber in the gallery.

We now show that all other group elements on the gallery are minimal right coset representatives.

Take a minimal right coset representative $s_{i_1}\dots s_{i_m}$ on the gallery $G_y$.  We claim that the preceding element on the gallery, $s_{i_2}\dots s_{i_m}$, is also a minimal length coset representative.
If not, then there is a $w\in S_n$ satisfying $\ell(s_{i_2}\dots s_{i_m}w)<\ell(w)+(m-1)$ (length is additive for minimal coset representatives in Coxeter groups; see Lemma~\ref{lem:coxeter.facts}\eqref{fact:min.coset.elts.add}).  But then $\ell(s_{i_1}s_{i_2}\dots s_{i_m}w)<\ell(w)+(m-1)+1$ since multiplying by a generator can add at most 1 to the minimal length.  That is, $\ell(s_{i_1}s_{i_2}\dots s_{i_m}w)<\ell(w)+m$, contradicting the minimality of $s_{i_1}\dots s_{i_m}$.
It follows that all elements corresponding to cones along a gallery $G_y$ are minimal coset representatives.

Part~\ref{prop:part1}.

It suffices to show that each $S_n$ coset contains an element in the gallery $G_y$.
Consider a point $Z=(Z_1,\dots,Z_n)$ in the fundamental cone that is also on the ray $E_y=\{y-\theta \mathbf j\mid \theta\in\R\}$.  The inequality its coordinates must satisfy is
\begin{equation}\label{eq:Z.ineq}
Z_n\ge\dots\ge Z_1\ge 0
\end{equation}
for $Z=(Z_1,\dots,Z_n)$.
Because the entries in $\j$ are all equal, this means $y_n\ge\dots\ge y_1\ge 0$ for $y=(y_1,\dots,y_n)$.

Moving along the ray $E_y$ in the positive direction (decreasing $\theta$) does not change the inequalities in Equation~\eqref{eq:Z.ineq} (all components stay positive) and hence the ray stays in the fundamental cone.  Increasing $\theta$ moves the ray through the gallery into different cones as first $Z_n\ge\dots\ge Z_2\ge 0\ge Z_1$, then $Z_n\ge\dots\ge Z_3\ge 0\ge Z_2\ge Z_1$ and so on.

On the other hand, acting by $S_n$ on a point on the ray permutes the entries, but this fixes the entries of
$\j$ and simply permutes the entries of $y$.  So a chamber on the gallery corresponds to an ordering of form $Z_n\ge\dots\ge Z_{i+1}\ge 0\ge Z_i\ge\dots\ge Z_1$, and the other chambers in its $S_n$ coset are obtained by permuting these entries.

Now consider an arbitrary point $v$ in $\R^n$ and denote the cone it is contained within by $C_{w(v)}$.
The action of $S_n$ permutes the entries of $v$, and there is a permutation that puts the entries in increasing order.  Every point that is in increasing order is in a cone that is on a ray-gallery, so we are done.
\end{proof}

We are now in a position to prove our main result.  Recall that this gives a generating function for the frequency distribution of intervals in the confidence net based on an irreducible finite Coxeter group.

\begin{proof}[Proof of Theorem~\ref{thm:main}]
The set of all rays $E_y$ from the identity chamber to the last chamber (labelled by $g$) gives the set of all possible galleries from 1 to $g$.  Each group element in each gallery is a minimum length coset representative of $S_n$ in $G$, and all $S_n$-cosets have their minimal length representative occuring in such a gallery (Proposition~\ref{prop:lemmas}).

Recall that $u^{-1}(j)$ is the set of chambers that are in the $j$'th position along a ray $E_y$.  When we start with the identity chamber, this is simply the set of group elements of length $j$ that appear in galleries.  From Proposition~\ref{prop:lemmas}, this is the set of minimal coset representatives of length $j$.  So the number $n_j$ is the number of minimal coset representatives of length $j$, and this is given by our formula, by~Lemma~\ref{lem:coxeter.facts}\eqref{fact:poincare.poly.quotient}.
\end{proof}

\section{Further examples}

In Sections~\ref{sec:typeB} and~\ref{sec:typeD} we gave examples of confidence nets from our theory for types $B$ and $D$ respectively, and recalculated them using Theorem~\ref{thm:main} after its statement.  Here we add types $E_6$, $E_7$, $E_8$ and $A_n$.  The remaining types of finite Coxeter group (omitted) are types $H_3$ and $H_4$, the dihedral groups $I_2(m)$, and the group $F_4$ (to which Theorem~\ref{thm:main} doesn't apply).

\subsection{Type $E$}
\label{sub:type.E}

Inserting the $d_j$ values for $E_6$, $E_7$ and $E_8$ from Table~\ref{tab:degrees} and obtaining help in factorization from Maple we have the following formulae:

\begin{eqnarray*}
E_6: & & (q+1)(q^2+1)(q^2-q+1)(q^4-q^2+1)(q^2+q+1)(q^6+q^3+1)(q^4+1) \\
& & \\
E_7: & &  (q+1)^4(q^2-q+1)^2(q^6+q^3+1)(q^6-q^3+1)\\
& & (q^6-q^5+q^4-q^3+q^2-q+1) \\
& & (q^2+1)(q^2+q+1)(q^4-q^2+1)(q^4-q^3+q^2-q+1)(q^4+1)\\
& & \\
E_8: & &  (q^4+q^3+q^2+q+1)(q^6+q^3+1)(q^6-q^3+1)(q^4+1) \\
& & (q^6-q^5+q^4-q^3+q^2-q+1)(q^8-q^7+q^5-q^4+q^3-q+1) \\
& & (q^8+q^7-q^5-q^4-q^3+q+1)(q^8-q^6+q^4-q^2+1)(q^8-q^4+1)\\
& & (q^2+q+1)^2(q^4-q^3+q^2-q+1)^2(q^2+1)^2 \\
& & (q^4-q^2+1)^2(q^2-q+1)^3(q+1)^4.
\end{eqnarray*}

Let us consider $E_8$ in a little more detail. It has order $2^{14} 3^5 5^2 7 = 696729600$, meaning that $\R^8$ is split into this many cones. The root system is described in the standard way as:
$$
\{\pm e_i \pm e_j : 1 \leq i < j\}, \;\;
\left\{ \frac{1}{2}\sum_{i=1}^8 \lambda_i e_i : \; \lambda_i = \pm 1, \prod_{i=1}^8 \lambda_i = 1 \right
\}.
$$ 
Again we have two ways of counting. The number of live roots are those not orthogonal to ${\bf j} = (1,1,1,1,1,1,1,1)^T$.
From the first set above we have those of the form $(1,1,0, \ldots)$, namely ${8 \choose 2}$. From the second set
we have all those for which the number of ones and zeros is different and even, being careful not to double count. This gives
$${8 \choose  2} + 1 + {8 \choose 2} + {8 \choose 6} = 92. $$
On the other hand $G(s) = a_0 +a_1q+ \ldots$ is a polynomial of degree 92 whose $N=93$ coefficients are laid out below to show the symmetry.
{\tiny
$$
  \begin{array}{rrrrrrrrrrr}
    1 & 1 & 1 & 2 & 3 & 6 & 6 & 8 & 10 & 13 & 17  \\
    21 & 26 & 32 & 38 & 46 & 55 & 64 & 74 & 86 & 98 & 112 \\
    127 & 142 & 157 & 175 & 193 & 211 & 230 & 249 & 267 & 287 & 307 \\
    325 &343 & 361 & 377 & 393 & 409 & 421 & 432 & 443 & 452 & 458 \\
    464 & 466 & 466 & 466 & 464 & & & & & & \\
    458 & 452 & 443 & 432 & 421 & 409 & 393 & 377 & 361 & 343 & 325 \\
    307 & 287 & 267 & 249 & 230 & 211 & 193 & 175 & 157 & 142 & 127 \\
    112 & 98 & 86 & 74 & 64 & 55 & 46 & 38 & 32 & 26 & 21  \\
    17 & 13 & 10 & 8 & 6 & 4 & 3 & 2 & 1 & 1 & 1 \\
  \end{array}
  $$
 }

\subsection{Type $A_n$} %
The usual interpretation of the action of type $A_n$ is as the restriction of the symmetric group $S_{n+1}$ to the hyperplane: $H:\sum_{i=1}^{n=1} x_i=0$. When $n=2$ this yields a figure in $2$-dimensions with cones with apex angle $\frac{1}{3}\pi$.  The role of $S_n$ in the above examples is now played by
$A_{n-1}$. Referring to the first entry in Table 1 this gives the generating function
\begin{eqnarray*}
G(q) & = & \frac{\prod_{i=1}^n (1-q^{i})}{ \prod_{i=1}^n (1-q^i)} \\
     & = & \sum_{i=0}^n q^n,
\end{eqnarray*}
giving a discrete uniform distribution on the net chambers.

The statistical interpretation takes a little care. There are different choices one can make for the
representation of $A_{n-1}$ as a subgroup of $A_n$. A simple choice is for $A_{n-1}$ to be the restriction to the hyperplane $H$
of the group that permutes the first $n$ coordinates. Let us require that the ``data" $Y$ also lies in $H$ and that the model
is given by
$$Y = \theta {\bf k} + Z$$
where $\mbox{prob} \{Z \in C_i\} = \frac{1}{(n+1)!}$ and the $C_i$ are the cones of $A_n$ in $H$ (with similar assumptions as in the introduction). The key is to
make the vector ${\bf k}$, which is the analogue of the previous ${\bf j}$, to be invariant under $A_{n-1}$. Thus, we can take
$${\bf k} = (1,1, \ldots, 1, -n)^T.$$
Following Lemma~\ref{lem:cones.ray}, we find the boundary of the net chamber by taking the intersection of the ray $Y-\theta {\bf k}$ with the live root of $A_n$ that is all those
not orthogonal to $\bf k$. These are
$$(1,0 \ldots, 0 -1)^T,(1,0 \ldots, 0 -1)^T , \ldots, (0,0 \ldots, 0, 1, -1)^T.$$
Taking the $j$-th member of this list first we see that the boundary
is given by
$$y_j-\theta - (y_{n+1} + n \theta) = y_j-y_{n+1} - (n+1) \theta.$$
But since $Y \in H$ we have $y_{n+1} = -\sum_{i=1}^n y_i$. This means that the boundaries
are the $n$ sample quantities
$$\frac{1}{n+1}\left(2 y_j + \sum_{i \neq j}^n x_i\right), \; j=1, \ldots, n,$$
giving $n+1$ chambers, as expected.

\section{A non-group cone example}
Exact coverage nets also arise for the situation in which $\R^n$ is divided into cones that are congruent, but not arising as the
fundamental cones of a reflection group. Consider the partition of the positive orthant into
$n$ cones generate by a ``long diagonal" and $n-1$ principal axes. Leaving out the first principal axis
we obtain generators:
$$(1,1,\ldots,1)^T, (0,1,0, \ldots, 0)^T, (0,0,1,\ldots, 0)^T, \ldots ,(0,0, \ldots ,1)^T.$$
The other cones are generated by successively omitting principal axes.
Now take all sign changes to reach all other quadrants. This divides $\R^n$ into $n 2^n$ congruent cones.

Assume the $Z$-probability content of each cone is equal and apply the method used for the other examples.
We first check how many, and which, walls are cut by a typical ray, and group together the cones which lead to the same ``index", as above. The number of planes is $2n+1$. After a little work
it turns out that the intervals formed by the order statistics $y_{(1)}< y_{(2)} < \ldots < y_{(n)}$
and all neighbour pairs $\frac{y_{(i)}+ y_{(j)}}{2}$ form a net of $2n$ intervals.

The successive net vectors (ignoring commas) are the rows below for $n=2, \ldots, 6$.
\vspace{3mm}
{\tiny
$$
\begin{array}{cccccccccccc}
       &  &   &  & 1 & 1 & 1 & 1 &  &   &  &  \\
       &  &   & 1 & 1 & 2 & 2 & 1 & 1 &   &  &  \\
       &  & 1 & 2 & 3 & 3 & 3 & 3 & 1 & 1 &  &   \\
       & 1& 1& 4 & 4 & 6 & 6 & 4 & 4 & 1 & 1 &  \\
     1 & 1& 5& 5 & 10&10 &10 & 10& 5 & 5 & 1 & 1 \\
\end{array}
$$
}

Note how each row is constructed by repeating the integer of the previous
row of the Pascal triangle
eg the row $1,6,15,20,15,6,1$ is split $6 \rightarrow (1,5),\;\;15 \rightarrow (5, 10),\;\;20 \rightarrow (10, 10)$,
giving the last row of the tableau above. The generating function is
$$(1+q)(1+q^2)^{n-1}.$$

\section{Some asymptotics}
The generating function for $B_n$ is well-known in the theory of partitions. It is the generating function for the number partition of an integer into at most $n$ distinct parts. The infinite version $G(q) = \prod (1+q^i)_{i=1}^{\infty}$ gives the number of partitions into $j$ distinct parts, with no other restrictions, and the two generating functions are identical up to $q^n$. The general  $G(q)$ has a long history. Following their celebrated work on partitions~\cite{hardy1918asymptotic},  Hardy and Ramanujan also studied this case, giving an asymptotic formula, see \cite{hua1942number} \cite{hardy1988inequalities}. For an extensive review see \cite{andrews1998theory}.

Noting the convergence of the Binomial distribution to the Normal (and following computer experimentation), it is natural to
conjecture that for $G_n(s)$ the $\{a_n\}$ follow an asymptotic distribution, and indeed this is the case.  The associated  probability distribution is that of the random variable $$U = \sum_{j=1}^{n} j V_j,$$
where the $V_i$ are iid Bernoulli random variables with probability $\frac{1}{2}$. Then, a theorem of H\'ajek and Sid\'ak \cite{hajek1967theory} for sums of
independent random variables with unequal means and variance gives $U \sim N(\mu, \sigma)$
where $\mu = \frac{1}{2}\sum_{j=1}^n = \frac{1}{4}n(n+1)$ and $\sigma^2 = \frac{1}{4} \sum_{j=1}^n j^2 = \frac{1}{24}n(n+1)(2n+1)$.

An Edgeworth-type expansion shows that the standardized random variable $\frac{U-\mu}{s}$ can be approximated
by $$\phi(u)\left(1 + \frac{\kappa_4}{24}\;\mathcal H_4(u)\right),$$
where $\phi$ is the standard Normal density, $\kappa_4$ is the fourth cumulant of the standardized variable, and $\mathcal H_4 =  u^4-6u^2+3$ is the order 4 standard Hermite polynomial.

After a little  work we derive, for $B_n$,
$$\kappa_4 = -\frac{12}{5} \frac{3n^2 +3n-1}{n(n+1)(2n+1)}.$$
Keeping the $\mbox{O}\left( \frac{1}{n} \right)$ terms we have the approximation
$$\phi(u)\left(1- \frac{3}{20} \mathcal H_4(u)\frac{1}{n} + \mbox{O}\left( \frac{1}{n^2} \right) \right). $$
There are similar result for $D_n$.

For $E_8$ the distribution  mean and variance  are $(\mu,\sigma^2) = (46,\frac{6811}{3})$ and probabilities roughly follow a normal distribution with this mean and variance. The approximation using $H_4$ is surprisingly good. For the standardized distribution
$$\frac{\kappa_4}{24}= -\frac{365311}{28896080} = -0.01264 \ldots.$$
Converting the approximation back to the original cell probabilities the maximum absolute deviation and the root mean squared
 error are approximately $1.5 \times 10^{-4}$ and $7.3 \times 10^{-5}$ respectively. For the Edgworth-type approximation, the integral (or in this case the sum) of the approximate probability will typically not be unity. In this case the sum of the approximands
is $1.0001534 \ldots$ so that the error is of the same order as the maximum deviation.

\section{Conclusions and further work}
This paper is a contribution to coverage problems in which, essentially, there are only group symmetry conditions on the underlying distribution. The most obvious limitation of the present paper is that it only covers a single parameter, although much classical non-parametrics is of this type. The long term aim is to use the ideas of this paper to develop coverage nets based on chambers in more than one dimension. Roughly, the requirements are (i) a large group that houses the distributional assumptions, (ii) a smaller sub-group under which the (linear) model is invariant, (iii) a valid quotient or coset operation, (iv) the use of the Chevalley factorization formula to perform the counting, and (v) more extensive use of the theory of buildings. Another challenge is to apply the theory to infinite groups such as affine Weyl groups which already have applications in physics, material science and genomics~\cite{Bodner2012affine}, \cite{prince1994mathematical}, \cite{egri2014group}. Finally, there may be theory  with upper and lower probabilities where exactness is hard to find which would lead to some kind of group-based belief functions.

\bibliographystyle{plain}

\end{document}